\documentclass[11pt,a4paper]{article}
\usepackage{a4,amssymb,amsmath,amsthm,url,color,enumerate}
\usepackage{bezier,amsfonts,amssymb,graphicx,amsthm,url}
\usepackage[english]{babel}
\title{Hamiltonicity and $\sigma$-hypergraphs}
\date{}
\begin{document}
\newtheorem{theorem}{Theorem}[section]
\newtheorem{definition}{Definition}[section]
\newtheorem{proposition}[theorem]{Proposition}
\newtheorem{corollary}[theorem]{Corollary}
\newtheorem{lemma}[theorem]{Lemma}
\newtheorem{problem}[theorem]{Problem}

\author{Christina Zarb \\Department of Mathematics \\University of Malta \\Malta }
\DeclareGraphicsExtensions{.pdf,.png,.jpg}

\maketitle

\begin{abstract}
We define and study a special type of hypergraph.  A $\sigma$-hypergraph $H= H(n,r,q$ $\mid$ $\sigma$), where $\sigma$ is a partition of $r$,  is an $r$-uniform hypergraph having $nq$ vertices partitioned into $ n$ classes of $q$ vertices each.  If the classes are denoted by $V_1$, $V_2$,...,$V_n$, then a subset $K$ of $V(H)$ of size $r$ is an edge if the partition of $r$ formed by the non-zero cardinalities $ \mid$ $K$ $\cap$ $V_i \mid$, $ 1 \leq i \leq n$, is $\sigma$. The non-empty intersections $K$ $\cap$ $V_i$ are called the parts of $K$, and $s(\sigma)$ denotes the number of parts.  We consider various types of cycles in hypergraphs such as Berge cycles and sharp cycles in which only consecutive edges have a nonempty intersection.  We show that most $\sigma$-hypergraphs contain a Hamiltonian Berge cycle and that, for $n \geq s+1$ and $q \geq r(r-1)$, a $\sigma$-hypergraph $H$ always contains a sharp Hamiltonian cycle.  We also extend this result to $k$-intersecting cycles.
\end{abstract}

\section{Introduction}

Let $V=\{v_1,v_2,...,v_n\}$ be a finite set, and let $E=\{E_1,E_2,...,E_m\}$ be a family of subsets of $X$.  The pair $H=(X,E)$ is called a \emph{hypergraph} with vertex-set $V(H)=V$, and with edge-set $E(H)=E$.  When all the subsets are of the same size $r$, we say that $H$ is an \emph{r-uniform hypergraph}.   A $\sigma$-hypergraph $H= H(n,r,q$ $\mid$ $\sigma$), where $\sigma$ is a partition of $r$,  is an $r$-uniform hypergraph having $nq$ vertices partitioned into $ n$ \emph{classes} of $q$ vertices each.  If the classes are denoted by $V_1$, $V_2$,...,$V_n$, then a subset $K$ of $V(H)$ of size $r$ is an edge if the partition of $r$ formed by the non-zero cardinalities $ \mid$ $K$ $\cap$ $V_i$ $\mid$, $ 1 \leq i \leq n$, is $\sigma$. The non-empty intersections $K$ $\cap$ $V_i$ are called the parts of $K$, and $s=s(\sigma)$ denotes the number of parts.  We denote the largest part of $\sigma$ by $\Delta=\Delta(\sigma)$ and the smallest part by $\delta=\delta(\sigma)$.  In order to avoid trivial situations where there are no edges, we shall always assume that $q \geq \Delta$ and $n \geq s$. These hypergraphs were first introduced in \cite{CaroLauri14} and studied further in \cite{CLz1,CLZ3}.

We consider Hamiltonian cycles in $\sigma$-hypergraphs.  In a graph $G$, a Hamiltonian path is a path which includes  every vertex $v \in V(G)$.  A Hamiltonian cycle is a closed Hamiltonian path.  It is well-known that the problem of determining whether a Hamiltonian  cycle exists in a graph is $NP$-complete. An excellent survey on results related to Hamiltonicity is given in \cite{gould2014}.

In hypergraphs, in particular $r$-uniform hypergraphs, there are several different types of paths and cycles to consider --- amongst the first to be defined was the \emph{Berge cycle} \cite{berge1984hypergraphs}.  A sequence $C=(v_1,e_1,v_2,e_2,\ldots,v_p,e_p,v_1)$  is a \emph{Berge cycle} if 
\begin{itemize}
\item{$v_1,v_2,\ldots,v_p$ are all distinct vertices}
 \item{$e_1,e_2,\ldots, e_p$ are all distinct edges}
\item{$v_k,v_{k+1} \in e_k$ for $k=1,\ldots,m$ where $v_{m+1}=v_1$}
\end{itemize} 

\noindent A Berge cycle is Hamiltonian if it covers all the vertices in the hypergraph.

Several other types of cycles and Hamiltonian cycles have been described and studied as in \cite{katona1999hamiltonian,2014arXiv1402.4268K,LeMatematiche114}.  The presentation \cite{katonacycles} gives an excellent survey of cycles and paths in hypergraphs.  We give the following definitions of cycles and Hamiltonian cycles which are particularly suited to the structure of $\sigma$-hypergraphs.

Consider an $r$-uniform hypergraph $H$.  Let $C=(e_1,\ldots,e_p)$ be a sequence of edges of $H$.  Then $C$ is a \emph{sharp cycle} if  $|e_i \cap e_{i+1}|>0$ for $1 \leq i \leq p$, where addition is modulo $p$, and $|e_i \cap e_j|=0$ otherwise.

A sharp cycle $C$ is a \emph{ sharp Hamiltonian cycle} if $V(C)=V(H)$.

A sharp Hamiltonian cycle $C=(e_1,e_2,\ldots,e_p)$  is said to be \emph{$(t,z)$-sharp} if $p=0\pmod 2$ and, for some $t,z > 0$,  $|e_i \cap e_{i+1}|=t$ when $i=1 \pmod 2$ and $|e_i \cap e_{i+1}|=z$ when $i=0\pmod 2$, for $1 \leq i \leq p$ and addition is modulo $p$.  If $t=z$, the cycle is \emph{$t$-sharp}.  These cycles are analogous to $t$-overlapping cycles as described in \cite{rucinski2013hamilton}.  A $1$-sharp cycle is often referred in the literature to as a \emph{loose cycle}.

Finally, a \emph{$k$-intersecting cycle} $C=(e_1,e_2,\ldots,e_p)$ is such that \[|e_i \cap e_{i+1} \cap \ldots \cap e_{i+k-1}| >0\] for $1 \leq i \leq p$, where addition is modulo $p$, while any other collection of $k$ or more edges has an empty intersection.  If $V(C)=V(H)$ then $C$ is a \emph{$k$-intersecting Hamiltonian cycle}.  Thus a sharp Hamiltonian cycle is a $2$-intersecting Hamiltonian cycle.

In this paper we consider all the above types of Hamiltonian cycles in $\sigma$-hypergraphs.  We first consider Berge cycles, and then move on to sharp Hamiltonian cycles, and finally to $k$-intersecting cycles.  We give some conditions for the existence and non-existence of the different types of Hamiltonian cycles in $\sigma$-hypergraphs, which then lead us to consider conditions on the parameters of $H=H(n,r,q \mid \sigma)$ for which the different types of Hamiltonian cycles always exist.

When constructing sharp Hamiltonian cycles,  we will use matchings --- the link between matchings and Hamiltonian cycles in $r$-uniform hypergraphs has been extensively studied \cite{berge1984hypergraphs}.  Given an $r$-uniform hypergraph $H$, a \emph{matching} is a set of pairwise vertex-disjoint edges $M \subset E(H)$.  A \emph{perfect matching} is a matching which covers all vertices of $H$ and we denote the size of the largest matching in an $r$-uniform hypergraph $H$ by $\nu(H)$.  

Matchings in $\sigma$-hypergraphs were studied in \cite{CLZ3} and, as in that paper we here give more structure to the vertices of the hypergraph $H=H(n,r,q \mid \sigma)$ with $\sigma=(a_1,a_2,\ldots,a_s)$, and $\Delta=a_1 \geq a_2 \geq \ldots \geq a_s=\delta$.   The classes making up the vertex set are ordered as $V_1,V_2,\ldots,V_n$ and, within each $V_i$, the vertices are ordered as $v_{1,i},v_{2,i},\ldots, v_{q,i}$.  We visualise the vertex set $V(H)$ as  a $q \times n$ grid whose first row is $v_{1,1},v_{1,2},\ldots,v_{1,n}$.  We sometimes refer to the vertices $v_{1,i},v_{2,i},\ldots, v_{k,i}$ as the top $k$ vertices of the class $V_i$, and to $v_{q-k+1,i},v_{q-k+2,i},\ldots, v_{q,i}$ as the bottom $k$ vertices of $V_i$.  The vertices $v_{k,i}$ and $v_{k+1,i}$ are said to be consecutive in $V_i$.  The class $V_1$ is called the first class of vertices, and $V_n$ is the last class; $V_i$ and $V_{i+1}$ are said to be consecutive classes.  A set of vertices contained in $h$ consecutive rows and $k$ consecutive classes of $V(H)$ is said to be an $h \times k$ subgrid of $V(H)$.  Also, for $\sigma=(a_1,a_2,\ldots,a_s)$, if $a_1=a_2=\ldots=a_s=\Delta$, $\sigma$ is said to be a \emph{rectangular} partition.  Furthermore, if $\sigma$ is rectangular and $\Delta=s(\sigma)$, then $\sigma$ is a \emph{square} partition.

We observe that it is well known that in a graph, the Hamiltonian cycle yields a perfect or near perfect (leaving one vertex unmatched if $n$ is odd) matching.  In \cite{CLZ3}, it was shown that there exist arbitrarily large $\sigma$-hypergraphs which do not have a perfect matching, and for which the number of unmatched vertices is quite large.  We state a result from this paper:

\begin{lemma} \label{not_r_good}
Let $H=H(n,r,q \mid \sigma)$, where $\sigma=(a_1,\ldots,a_s)$, $n \geq s$ and $ q \geq r$.  Suppose $gcd(\sigma)=d \geq 2$, and $q=t\pmod d$ where $1 \leq t \leq d-1$. Then in a maximum matching of $H$, there are at least $tn$ vertices left unmatched.  Hence $\nu(H) \leq \frac{n(q-t)}{r}$.
\end{lemma}

In the sequel, we will show that these $\sigma$-hypergraphs, however, still have both a Berge and a sharp Hamiltonian cycle when $q$ and $n$ are large enough.

\section{Berge Cycles}

Let us first consider this type of cycle, and give necessary and sufficient conditions for the existence of Hamiltonian Berge cycles  in $\sigma$-hypergraphs.

\begin{theorem}
Let $H=H(n,r,q \mid \sigma)$ with $\sigma=(a_1,a_2,\ldots,a_s)$, $s \geq 2$ and $\Delta=a_1 \geq a_2 \geq \ldots \geq a_s=\delta\geq1$.  If $\sigma$ is not rectangular and $q \geq \Delta$ and $n \geq s$, then $H$ has Hamiltonian Berge cycle.  If $\sigma$ is rectangular, then there is a Berge Hamiltonian cycle when $q \geq \Delta+1$ or $n \geq s+1$.
\end{theorem}

\begin{proof}
Let us take a partition $\sigma$ which is not rectangular --- we construct a Berge cycle as follows:  for $e_1$, we take the bottom $a_1$ vertices in $V_1$, the bottom $a_2$ vertices in $V_2$ and so on up to the bottom $a_s$ vertices in $V_s$.  For $e_{2}$, we ``shift" the edge one column to the right so that the parts are taken from $V_2$ to $V_{s+1}$.  We carry on in this fashion, and when we take part $a_{s-1}$ from $V_n$ then we take part $a_s$ from $V_1$, but ``shift" one row up.  We carry on in this way, shifting one column to the right each time, and shifting one row up each time --- when we reach the top row then we start using the bottom vertices again.  In all, we form $nq$ distinct edges in this way.  We can now order the vertices by labelling the bottom vertex in $V_1$ as $v_1$, the bottom vertex in $V_2$ as $v_2$, and so on upto the bottom vertex in $V_n$ as $v_n$ --- we then move up one row and label the vertices in this row $v_{n+1}$ up to $v_{2n}$, from left to right, and we carry on in this fashion until we have labelled all vertices in this way.  

Then the cycle $v_1,e_1,v_2,e_2,\ldots, v_{nq},e_{nq},v_1$ is a Hamiltonian Berge cycle.  

If $\sigma$ is rectangular and $q=\Delta$ and $n=s$, then there is only one edge and hence no Berge Hamiltonian cycle, otherwise the Berge Hamiltonian cycle can be constructed as above.

\end{proof}

Note:  The conditions are easily seen to be necessary, that is if $H$ has a Berge Hamiltonian cycle then necessarily $q \geq \Delta$ and $n \geq s$ otherwise $H$ has no edges, while when $\sigma$ is rectangular, $q$  must be at least $\Delta+1$ .  
\medskip

Figure \ref{berge} gives an example of a Hamiltonian Berge cycle for $H=H(n,r,q \mid \sigma)$ with $\sigma=(2,1)$, $s=2$ and $q=n=3$.  The cycle has 9 edges.  The shaded vertices form the edge in each case, and the vertices are numbered in cyclic order.

\begin{figure}[h!]
\centering
\includegraphics{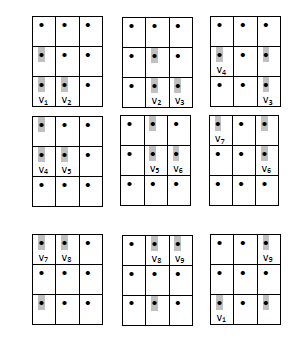} 
\caption{Berge Hamiltonian Cycle - shaded vertices represent the linking edges} \label{berge}
\end{figure}

\section{Sharp Hamiltonian Cycles}

We have  given necessary and sufficient conditions for a $\sigma$-hypergraph to contain a Berge Hamiltonian cycle.  Hence we now turn our attention to sharp Hamiltonian cycles which prove to be more challenging.  Although we shall be studying, in a later section, $k$-intersecting Hamiltonian cycles, we want to treat separately sharp cycles first, which are $k$-intersecting for $k=2$, because they illustrate very clearly the main techniques used in this paper and also because stronger results are possible with sharp cycles, when, in some cases, we prove that the Hamiltonian cycles obtained are either $t$-sharp or $(t,z)$-sharp.



We first present some basic observations about sharp Hamiltonian cycles in $r$-uniform hypergraphs.

\begin{lemma} \label{fact1}
Let $H$ be a $r$-uniform hypergraph and let $C$ be a sharp Hamiltonian cycle in $H$.  Then
\begin{enumerate}
\item{ $\frac{2|V(H)|}{r} \geq |E(C)| \geq \frac{|V(H)|}{r-1}$}
\item{$\nu(H) \geq \nu(C) = \left \lfloor \frac{|E(C)|}{2} \right \rfloor.$}
\item{If  $2 \nu(H) +1 < \frac{nq}{r-1}$,  there is no sharp Hamiltonian cycle in $H$.}
\end{enumerate}
\end{lemma}

\begin{proof}
\mbox{}\\*

\noindent 1. \indent Consider $C=(e_1,e_2,\ldots,e_p)$.  Clearly, each edge of $C$ intersects the next edge, so each edge contributes at most $r-1$ vertices to $V(C)=V(H)$, and hence $|V(H)| \leq p(r-1)$ which imples $p=|E(C)| \geq \frac{|V(H)|}{r-1}$.

Now consider the degrees of the vertices in $C$.  No vertex can have degree greater than 2, so using the well known fact that \[r|E(C)|=\sum deg_C(v) \leq 2|V(H)|\] we get the result $|E(C)| \leq \frac{2|V(H)|}{r}$.

If we apply this  to $\sigma$-hypergraphs, where $V(H)=nq$, we get \[ \frac{2nq}{r} \geq |E(C)| \geq \frac{nq}{r-1}.\]
\medskip

\noindent 2. \indent  By the definition of a sharp cycle the  subset of $E(C)$ $\{e_{2j+1} : 0 \leq j \leq \left \lfloor \frac{|E(C)|}{2} \right \rfloor\}$ is a maximal matching in $C$, as the edges are distinct, and clearly any other edge in $C$ will intersect one of these edges.  Hence \[\nu(H) \geq \nu(C) = \left \lfloor \frac{|E(C)|}{2} \right \rfloor.\]
\medskip

\noindent 3. \indent Clearly, by part 1 of this lemma and the assumption $2 \nu(H) +1 < \frac{nq}{r-1}$, \[|E(C)| \geq \frac{nq}{r-1} > 2\nu(H)+1 \geq 2\nu(C)+1.\]  But $|E(C)|$ is an integer hence 
\begin{eqnarray*}
|E(C)| &\geq& 2\nu(H)+2 \geq 2\nu(C)+2 \\
&=& 2 \left \lfloor \frac{|E(C)|}{2} \right \rfloor + 2 \geq 2 \left( \frac{|E(C)|-1}{2} \right ) + 2\\
&=& |E(C)|+1,
\end{eqnarray*}

 a contradiction.
\end{proof}

\subsection{Examples of $\sigma$-hypergraphs with no Hamiltonian cycle.}

Let us consider examples of $\sigma$-hypergraphs in which there is no sharp Hamiltonian cycle.

For the first example we use Lemma \ref{fact1} --- consider $H=H(n,r,q \mid \sigma)$ with $\sigma=(\Delta,\Delta,\ldots,\Delta)$ and  $s(\sigma)=\Delta \geq 2$, that is $\sigma$ is a square partition.  Let $n=q=2\Delta-1$.

In this case, it is easy to see that $\nu(H)=1$, while $\frac{nq}{r-1}=\frac{(2\Delta-1)^2}{\Delta^2-1}$.  Hence, if $\frac{(2\Delta-1)^2}{\Delta^2-1} > 3$, that is for $\Delta \geq 3$, there is no sharp Hamiltonian cycle in $H$.

As a second example, consider $H=H(n,r,q \mid \sigma)$ with $\sigma=(\Delta,1,\ldots,1)$ where $s(\sigma)=\Delta= \frac{r+1}{2} \geq 4$ and $q=n=\frac{r+1}{2}$.  Consider an edge $E_1$ with the first part of size $\Delta$ taken from $V_1$, and the other parts taken from $V_2$ to $V_n$ respectively.  It is clear that any other edge intersects this edge --- we need at least four edges for a sharp Hamiltonian cycle, but this is impossible since all these edges intersect $E_1$ and hence the cycle is not sharp.
\bigskip

\noindent In view of the above examples, our goal is to deal with the following problem which now arises naturally:
\begin{problem}\label{mainresult1}
Let $H=H(n,r,q \mid \sigma)$, where $s(\sigma) \geq 2$.  Does there exists $q(\sigma)$ and $n(\sigma)$ such that $\forall q \geq q(\sigma)$ and $n \geq n(\sigma)$, $H$ has a sharp Hamiltonian cycle?
\end{problem}

In the sequel we will supply an affirmative solution to this problem.  So we begin with some results which will then allow us to find a solution to this problem.

\begin{lemma} \label{rdivq}
Let $H=H(n,r,q \mid \sigma)$ with $\sigma=(a_1,a_2,\ldots,a_s)$, $s \geq 2$ and $\Delta=a_1 \geq a_2 \geq \ldots \geq a_s=\delta\geq1$.   Let $1 \leq p <s$, and let \[ t=\sum_{i=1}^{i=p}a_i \mbox{   and   } z=\sum_{i=p+1}^{i=s}a_i =r-t.\] If $r | q$ and $n \geq s+1$, then $H$ has a $(t,z)$-sharp Hamiltonian cycle. Moreover if there exists $k$ such that \[ t=\sum_{i=1}^{i=p}a_i =\sum_{i=p+1}^{i=s}a_i ,\] then the cycle is $t$-sharp.
\end{lemma}
\begin{proof}
Let us consider the first $r$ vertices in $V_1,\ldots,V_ n$ as an $r \times n$ grid of vertices.  We will construct  two perfect matchings $M$ and $M^*$, whose edges will then be used to form a sharp Hamiltonian cycle $C_1$ with $2n$ edges.  Observe that we need $s=s(\sigma) \geq 2$, otherwise if $s=1$, that is $\sigma=(r)$, then for $n \geq 2$ $H$ is not connected, while for $n=1$, $H$ is the complete $r$-uniform hypergraph on $q$ vertices, which is trivially Hamiltonian for $q \geq r+1$.

For the first matching $M$, let each column of $r$ vertices be partitioned into $s$ consecutive parts of sizes $a_1,a_2,\ldots,a_s$.  The part $a_i$ in $V_j$ will be referred to as the $i^{th}$ part in $V_j$.   The edge  $E_1$ is formed by taking the top  $a_1$ vertices from $V_1$, the second part of size $a_2$ from $V_2$  and so on, ``in diagonal fashion".  This is repeated for $E_2$ by ``shifting one class to the right", taking the top $a_1$ vertices from $V_2$, the second part from $V_3$ etc.  In general, the edge $E_j$, $1 \leq j \leq n$, takes the first part from $V_j$, the second part from $V_{j+1}$ and in general the $k^{th}$ part from $V_{j+k-1}$, for $1 \leq k \leq s$, with addition modulo $n$.  This gives a perfect matching $M$ with $n$ edges.

For the second matching $M^*$, let $1 \leq p <s$, and let \[ t=\sum_{i=1}^{i=p}a_i \mbox{   and   } z=\sum_{i=p+1}^{i=s}a_i =r-t\]  Then we take edge $E_1^*$ such that parts $a_1$ to $a_{p}$ are taken as in edge $E_1$, while part $a_{p+1}$ to $a_s$ are taken as in edge $E_2$ - this is possible since $n \geq s+1$ and hence $E_1^*$ is different from all $E_i$ in $M$.  In general, $E_i^*$ has parts $a_1$ to $a_{p}$ as in edge $E_i$, and parts $a_{p+1}$ to $a_s$ as in edge $E_{i+1}$, where addition is modulo $n$.  It is clear that these edges form another perfect matching, and that they are distinct from the edges taken in $M$.

Now the sharp Hamiltonian cycle $C_1$ is formed by taking the edges in $M \cup M^*$ in this order: \[E_1,E_1^*,E_2, E_2^*,\ldots,E_{i},E_{i}^*,E_{i+1},E_{i+1}^*,\ldots,E_n,E_n^*.\]  In general, $|E_{i} \cap E_{i}^*|=a_1+\ldots+a_{p}=t$ and $|E_{i}^* \cap E_{i+1}|=a_{p+1}+ \ldots + a_s=z=r-t$, and at the end $E_n^* $ has parts $a_{p+1}$ to $a_s$ to coincide with the same parts in $E_1$ to close the cycle.  It is clear that the cycle is $(t,z)$-sharp since the only intersections between edges are the ones described.  If there exists $p$ such that \[ t=\sum_{i=1}^{i=p}a_i =\sum_{i=p+1}^{i=s}a_i =z=r-t,\] then the cycle is $t$-sharp.

Now if $q \geq r$, then for the next $r$ vertices in $V_1$ to $V_n$ we can create another cycle $C_2$ in the same way.  To link $C_1$ to $C_2$, we take  parts $a_{p+1}$ to $a_s$ of the last edge of $M^*$ in $C_1$ to coincide with the same parts in the first edge of $M$  in $C_2$.  

Hence if $q=xr$, we have a sharp Hamiltonian cycle $C=C_1 \cup C_2 \cup \ldots \cup C_x$, with the last edge of $C_x$ intersecting the first edge of $C_1$ in the  parts $a_{p+1}$ to $a_s$ of $E_1$, thus forming a $(t,z)$-sharp Hamiltonian cycle in $H$.

Now if there exists $k$ such that \[ t=\sum_{i=1}^{i=p}a_i =\sum_{i=p+1}^{i=s}a_i ,\] it is clear that the cycle is $t$-sharp.

\end{proof}

\begin{lemma} \label {r+1divq}
Let $H=H(n,r,q \mid \sigma)$ with $\sigma=(a_1,a_2,\ldots,a_s)$, $s \geq 2$ and $\Delta=a_1 \geq a_2 \geq \ldots \geq a_s=\delta\geq1$.  Let $1 \leq p <s$, and let \[ t=\sum_{i=1}^{i=p}a_i \mbox{   and   } z=\sum_{i=p+1}^{i=s}a_i =r-t.\] If $r+1 | q$ and $n \geq s+1$, then $H$ has an $(t-1,z)$-sharp Hamiltonian cycle. 
\end{lemma}

\begin{proof}
Let us first consider the first $r+1$ vertices in $V_1,\ldots,V_ n$ as an $(r+1) \times n$ grid of vertices.  As in the previous Lemma, we will construct  two matchings $M$ and $M^*$, whose edges can be used to form a sharp Hamiltonian cycle $C_1$ with $2n$ edges.

The first matching $M$ is constructed in exactly the same way as $M$ was constructed in Lemma \ref{rdivq} to cover the top $r \times n$ grid of vertices, having $n$ edges and  leaving out the vertices in the $(r+1)^{th}$ row.

For the second matching $M^*$, again we take $1 \leq p <s$, and let \[ t=\sum_{i=1}^{i=p}a_i \mbox{   and   } z=\sum_{i=p+1}^{i=s}a_i =r-t.\] The edges are then formed as follows: for edge $E_1^*$, part $a_1$ is taken as in $E_1$, but replacing the last vertex in this part with the $(r+1)^{th}$ vertex in the same class.  Parts $a_2$ to $a_{p}$ are taken as per edge $E_1$, while parts $a_{p+1}$ to $a_s$ are taken as per edge $E_2$.  Therefore, in general, $E_i^*$ has part $a_1$ taken from $V_i$ to include the top $a_1-1$ vertices, and the last vertex in the class, parts $a_2$ to $a_{p}$ as per edge $E_i$ , and parts $a_{p+1}$ to $a_s$ as in edge $E_{i+1}$.

Now the sharp Hamiltonian cycle is $C_1=(E_1,E_1^*,E_2,E_2^*,\ldots,E_n,E_n^*)$ so that $|E_i \cap E_i^*|=t-1$, and $|E_i^* \cap E_{i+1}|=z=r-t$.  The last edge $E_n^*$ intersects $E_1$ in parts $a_{p+1}$ to $a_s$.

Now if $q \geq r+1$ and $(r+1)|q$, then for the next $r+1$ vertices in $V_1$ to $V_n$ we can create another cycle $C_2$ in the same way.  To link $C_1$ to $C_2$, we take the parts $a_{p+1}$ to $a_s$ for $E_n^*$ to coincide with the same parts in the first edge in $C_2$.  

Hence if $q=x(r+1)$, we have a sharp Hamiltonian cycle $C=C_1 \cup C_2 \cup \ldots \cup C_x$, with the last edge of $C_x$ intersecting the first edge of $C_1$ in the parts $a_{p+1}$ to $a_s$  of $E_1$, thus forming a $(t-1,z)$-sharp Hamiltonian cycle in $H$.

Again, if there exists $p$ such that \[ t-1= \left(\sum_{i=1}^{i=p}a_i \right)-1 =\sum_{i=p+1}^{i=s}a_i,\] then the cycle is $(t-1)$-sharp.
\end{proof}

We shall use the following classical theorem by Frobenius which states:

\begin{theorem} \label{Frobenius}
Let $a_1,a_2 $ be positive integers with $gcd(a_1,a_2)=1$.  Then for  $n \geq (a_1-1)(a_2-1)$, there are nonnegative integers $x$ and $y$ such that  $x a_1 +y a_2=n$.
\end{theorem}

Using Lemmas \ref{rdivq} and \ref{r+1divq} combined with Theorem \ref{Frobenius}, we can now present an affirmative solution to Problem \ref{mainresult1}, which we restate as a Theorem:

\begin{theorem} \label{SHC}
Let $H=H(n,r,q \mid \sigma)$, where $s(\sigma) \geq 2$.  If $q\geq r(r-1)$ and $n\geq s+1$, then $H$ has a sharp Hamiltonian cycle.
\end{theorem}

\begin{proof}

By Theorem \ref{Frobenius}, we know that if $q \geq  r(r-1)$, there exist nonnegative integers $x$ and $y$ such that  $x r +y(r+1)=q$, since $r$ and $r+1$ are always coprime.  So let us divide the $q \times n$ grid into $x$ consecutive grids of size $r \times n$, followed by $y$ consecutive grids of size $(r+1) \times n$.  If we consider the $xr \times n$ grid first, we know that by Lemma \ref{rdivq}, there is a sharp Hamiltonian cycle $C_1$ covering these vertices, and by Lemma \ref{r+1divq}, there is a sharp Hamiltonian cycle $C_2$ covering the $y(r+1) \times n$ grid.  If $x=0$ or $y=0$, then $C_1$, respectively $C_2$ give the required sharp Hamiltonian cycle. So we may assume that both $x$ and $y$ are greater than $0$.  We now need to look at linking $C_1$ to $C_2$ and viceversa.  Firstly, instead of linking the last edge in $C_1$ with the first one, we link it to the first edge in $C_2$, by taking the parts $a_{k+1}$ to $a_s$ for this last edge to coincide with the parts $a_{k+1}$ to $a_s$ in the first edge of $C_2$.  The last edge of $C_2$,  must be linked to the first edge in $C_1$.  So we take the parts $a_{k+1}$ to $a_s$ of the last edge in $C_2$ to coincide with these parts in the first edge in $C_1$.  Thus $C_1 \cup C_2$ form a sharp Hamiltonian cycle in $H$.
\end{proof}

\section{$k$-intersecting Hamiltonian cycles}

We now turn to $k$-intersecting cycles and generalise the results obtained in the previous section to $k$-intersecting Hamiltonian cycles in $\sigma$-hypergraphs.  Recall that a \emph{$k$-intersecting cycle} $C=(e_1,e_2,\ldots,e_p)$ is such that \[|e_i \cap e_{i+1} \cap \ldots \cap e_{i+k-1}| >0\] for $1 \leq i \leq p$, where addition is modulo $p$, while any other collection of $k$ or more edges has an empty intersection.  If $V(C)=V(H)$ then $C$ is a \emph{$k$-intersecting Hamiltonian cycle}.

\begin{lemma} \label{kint1}
Let $H=H(n,r,q \mid \sigma)$ with $\sigma=(a_1,a_2,\ldots,a_s)$, $s \geq 2$ and $\Delta=a_1 \geq a_2 \geq \ldots \geq a_s=\delta\geq1$.  Let $2 \leq k \leq s$.  If $r | q$ and $n \geq s+1$, then $H$ has a $k$-intersecting Hamiltonian cycle.
\end{lemma}

\begin{proof}
Let us consider the first $r$ vertices in $V_1,\ldots,V_ n$ as an $r \times n$ grid of vertices.  We construct $k$ perfect matchings $M_1,\ldots,M_k$ which we then use to construct a $k$-intersecting Hamiltonian cycles.  Recall that $2 \leq k \leq s$, and $k=2$ is equivalent to a sharp Hamiltonian cycle.

The first matching $M_1$ is equivalent to the matching $M$ in Lemma \ref{rdivq}, with edges labelled as $E_{1,1}, E_{1,2},\ldots,E_{1,n}$.

In the matching $M_2$, we take edges $E_{2,1}$ to $E_{2,n}$ so that edge $E_{2,i}$ has parts $a_1$ to $a_{k-1}$ as in edge $E_{1,i}$, while parts $a_k$ to $a_s$ are as in edge $E_{1,i+1}$.

In the matching $M_3$, we take edges $E_{3,1}$ to $E_{3,n}$ so that edge $E_{3,i}$ has parts $a_1$ to $a_{k-2}$ as in edge $E_{1,i}$, while parts $a_{k-1}$ to $a_s$ are as in edge $E_{1,i+1}$.

In general, in the matching $M_j$, we take edges $E_{j,1}$ to $E_{j,n}$ so that edge $E_{j,i}$ has parts $a_1$ to $a_{k-j+1}$ as in edge $E_{1,i}$, while parts $a_{k-j+2}$ to $a_s$ are as in edge $E_{1,i+1}$, for $2 \leq j \leq k$.  Since $n \geq s+1$, this is possible for all values of $k$ between $2$ and $s$.

Now we form a $k$-intersecting Hamiltonian cycle $C_1$ by taking the edges in the following order 
\medskip

\noindent $E_{1,1},E_{2,1},\ldots,E_{k,1},E_{1,2}\ldots,E_{k,2},\ldots,E_{1,i},E_{2,i},\ldots,E_{k,i},E_{1,i+1},\ldots,E_{k,i+1},\ldots$

\noindent $E_{1,n},\ldots,E_{k,n}$.
\medskip

We now look at the intersections:

$E_{1,1},E_{2,1},\ldots,E_{k,1}$ intersect in part $a_1$.

$E_{2,1},E_{3,1},\ldots,E_{1,2}$ intersect in parts $a_k$ to $a_s$.

$E_{3,1},E_{4,1},\ldots,E_{2,2}$ intersect in part $a_{k-1}$.

In general, $E_{j,i},E_{j+1,i},\ldots,E_{j-1,i+1}$ intersect in part $a_{k-j+2}$ for $3 \leq j \leq k$.  The last $k-1$ edges intersect $E_{1,1}$ in parts with the same intersections described above, making this cycle a $k$-intersecting Hamiltonian cycle.

Now if $q \geq r$ and $r|q$, then for the next $r$ vertices in $V_1$ to $V_n$ we can create another cycle $C_2$ in the same way.  To link $C_1$ and $C_2$, we must consider the last $k-1$ edges taken in $C_1$, that is edge $E_{2,n}$ to $E_{k,n}$.  For edge $E_{2,n}$,we take parts $a_k$ to $a_s$ to coincide with the same parts in the first edge in $C_2$, and in general, for edge $E_{j,n}$ we take parts $a_{k-j+2}$ to $a_s$ to coincide with the same parts in the first edge in $C_2$.

Hence if $q=pr$, we have a sharp Hamiltonian cycle $C=C_1 \cup C_2 \cup \ldots \cup C_p$, with the last $k-1$ edges of $C_p$ intersecting the first edge of $C_1$ in the  respective parts as described for $C_1$ intersecting $C_2$. 
\end{proof}
\medskip

Figure \ref{4int} shows the first two edges in the four perfect matchings required  for a $4$-intersecting Hamiltonian cycle when $\sigma=(a_1,a_2,\ldots,a_6)$, $q=r$ and $n=7$.  The boxes represent the parts $a_1$ to $a_6$ ordered from top to bottom --- $E_{j,1}$ is shaded in light grey while $E_{j,2}$ is shaded in dark grey, for $1 \leq j \leq 4$.   Each matching has seven distinct edges.

\begin{figure}[h!]
\centering
\includegraphics{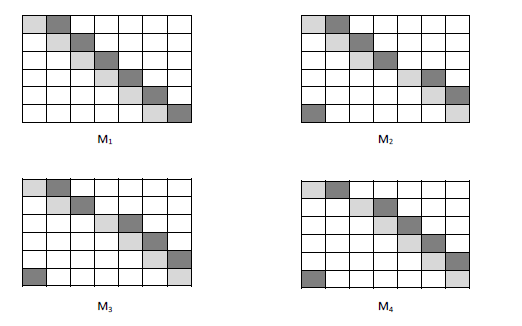} 
\caption{$\sigma= (a_1,a_2,\ldots,a_6)$ - matchings $M_1$ to $M_4$} \label{4int}
\end{figure}

\begin{lemma} \label{kint2}
Let $H=H(n,r,q \mid \sigma)$ with $\sigma=(a_1,a_2,\ldots,a_s)$, $s \geq 2$ and $\Delta=a_1 \geq a_2 \geq \ldots \geq a_s=\delta\geq1$.  Let $2 \leq k \leq s$.  If $(r+1) | q$ and $n \geq s+1$, then $H$ has a $k$-intersecting Hamiltonian cycle.
\end{lemma}

\begin{proof}
Let us consider the first $r+1$ vertices in $V_1,\ldots,V_ n$ as an $(r+1) \times n$ grid of vertices.  We construct $k$ matchings $M_1,\ldots,M_k$ which we then use to construct a $k$-intersecting Hamiltonian cycle, using a method similar to that used in the previous lemma.  

The first matching $M_1$ is for the top $r$ vertices in $V_1$ to $V_n$, and is equivalent to the matching $M$ in Lemma \ref{kint1}, with edges labelled as $E_{1,1}, E_{1,2},\ldots,E_{1,n}$.

In the matching $M_2$ for the $(r+1) \times n$ grid , we take edges $E_{2,1}$ to $E_{2,n}$ so that edge $E_{2,i}$ has part $a_1$ as in $E_{1,i}$, but replacing the last vertex in this part with the $(r+1)^{th}$ vertex in the same class, parts $a_2$ to $a_{k-1}$ as in edge $E_{1,i}$, while parts $a_k$ to $a_s$ are as in edge $E_{1,i+1}$.
 
In the matching $M_3$, we take edges $E_{3,1}$ to $E_{3,n}$ so that edge $E_{3,i}$ has part $a_1$ as in edge $E_{2,i}$, parts $a_2$ to $a_{k-2}$ as in edge $E_{1,i}$, while parts $a_{k-1}$ to $a_s$ are as in edge $E_{1,i+1}$.

In general, in the matching $M_j$, we take edges $E_{j,1}$ to $E_{j,n}$ so that edge $E_{j,i}$ has parts $a_1$ as in edge $E_{2,i}$, parts $a_2$ to $a_{k-j+1}$ as in edge $E_{1,i}$, while parts $a_{k-j+2}$ to $a_s$ are as in edge $E_{1,i+1}$, for $3 \leq j \leq k$.

Now we form a $k$-intersecting Hamiltonian cycle $C_1$ by taking the edges in the following order:

\noindent $E_{1,1},E_{2,1},\ldots,E_{k,1},E_{1,2}\ldots,E_{k,2},\ldots,E_{1,i},E_{2,i},\ldots,E_{k,i},E_{1,i+1},\ldots,E_{k,i+1},\ldots$

\noindent $E_{1,n},\ldots,E_{k,n}$.
\medskip

We now look at the intersections:

$E_{1,1},E_{2,1},\ldots,E_{k,1}$ intersect in $a_1-1$ vertices in part $a_1$.

$E_{2,1},E_{3,1},\ldots,E_{1,2}$ intersect in parts $a_k$ to $a_s$.

$E_{3,1},E_{4,1},\ldots,E_{2,2}$ intersect in part $a_{k-1}$.
\medskip

In general, $E_{j,i},E_{j+1,i},\ldots,E_{j-1,i+1}$ intersect in part $a_{k-j+2}$ for $3 \leq j \leq k$.  The last $k-1$ edges intersect $E_{1,1}$ in parts with the same intersections described above,  making this cycle a $k$-intersecting Hamiltonian cycle.

Now if $q \geq r+1$ and $(r+1)|q$, then for the next $r+1$ vertices in $V_1$ to $V_n$ we can create another cycle $C_2$ in the same way.  To link $C_1$ and $C_2$, we must consider the last $k-1$ edges taken in $C_1$, that is edge $E_{2,n}$ to $E_{k,n}$.  For edge $E_{2,n}$,we take parts $a_k$ to $a_s$ to coincide with the same parts in the first edge in $C_2$, and in general, for edge $E_{j,n}$ we take parts $a_{k-j+2}$ to $a_s$ to coincide with the same parts in he first edge in $C_2$.

Hence if $q=p(r+1)$, we have a sharp Hamiltonian cycle $C=C_1 \cup C_2 \cup \ldots \cup C_p$, with the last $k-1$ edges of $C_p$ intersecting the first edge of $C_1$ in the  respective parts as described for $c_1$ intersecting $C_2$. 
\end{proof}

We can now prove a generalised form of Theorem \ref{SHC}:

\begin{theorem}
Let $H=H(n,r,q \mid \sigma)$, where $s(\sigma) \geq 2$.  For $2 \leq k \leq s$,  if $q\geq r(r-1)$ and $n\geq s+1$, then $H$ has a $k$-intersecting  Hamiltonian cycle.
\end{theorem}
 \begin{proof}
By Theorem \ref{Frobenius}, we know that if $q \geq  r(r-1)$, there exist nonnegative integers $x$ and $y$ such that  $x r +y(r+1)=q$, since $r$ and $r+1$ are always coprime.  So let us divide the $q \times n$ grid into $x$ consecutive grids of size $r \times n$, followed by $y$ consecutive grids of size $(r+1) \times n$.  If we consider the $xr \times n$ grid first, we know that by Lemma \ref{kint1}, there is a $k$-intersecting Hamiltonian cycle $C_1$ covering these vertices, and by Lemma \ref{kint2}, there is a $k$-intersecting Hamiltonian cycle $C_2$ covering the $y(r+1) \times n$ grid.  If $x=0$ or $y=0$, then $C_1$, respectively $C_2$ give the required $k$-intersecting Hamiltonian cycle. So we may assume that both $x$ and $y$ are greater than $0$.  We now need to look at linking $C_1$ to $C_2$ and viceversa.  Firstly, instead of linking the last $k-1$ edges  in $C_1$ with the first one, we link them to the first edge in $C_2$, as follows:  for edge $E_{2,n}$,we take parts $a_k$ to $a_s$ to coincide with the same parts in the first edge in $C_2$, and in general, for edge $E_{j,n}$ we take parts $a_{k-j+2}$ to $a_s$ to coincide with the same parts in the first edge in $C_2$.

The last $k-1$ edges of $C_2$,  must be linked to the first edge in $C_1$.  So we take the respective parts in these edges to coincide with these parts in the first edge in $C_1$, in the same way we linked the last $k-1$ edges in $C_1$ to the first edge in $C_2$.  Thus $C_1 \cup C_2$ form a sharp Hamiltonian cycle in $H$.
\end{proof}

\section{Conclusion}

The paper \cite{CaroLauri14} defined $\sigma$-hypergraphs and started their study in order to investigate what are known as mixed colourings or Voloshin colouring \cite{voloshin02} of hypergraphs.  In the colourings in \cite{CaroLauri14}, no edge was allowed to have all vertices having the same colour, or all vertices having different colours.  This study was continued in  \cite{CLz1}.  These papers demonstrated the versatility of $\sigma$-hypergraphs in obtaining interesting results on mixed colourings.  In \cite{CLZ3},  the study of $\sigma$-hypergraphs was extended to two other classical areas of graph and hypergraph theory:  matchings and independence.  In this paper we continue in this vein, showing that $\sigma$-hypergraphs can also give elegant results on Hamiltonicity.

\bibliographystyle{plain}
\bibliography{z1bibv1}

\end{document}